\documentclass[11pt,leqno]{amsart}
\usepackage{amsmath,amssymb,amsthm}
\usepackage{hyperref}
\usepackage{bbm}

\newcommand{\R}{\mathbb{R}}

\newcommand{\ep}{\varepsilon}

\DeclareMathOperator{\co}{co}

\renewcommand{\geq}{\geqslant}
\renewcommand{\leq}{\leqslant}

\newcommand{\restricted}{\mathord{\upharpoonright}}

\DeclareMathOperator{\spann}{span}

\newcommand{\Lip}{{\mathrm{Lip}}_0}

\newcommand{\supp}{\operatorname{supp}}
\newcommand\restr[2]{\ensuremath{\left.#1\right|_{#2}}}

\newtheorem{Theorem}{Theorem}
\newtheorem{theorem}{Theorem}[section]
\newtheorem{lemma}[theorem]{Lemma}

\newtheorem{proposition}[theorem]{Proposition}
\newtheorem{corollary}[theorem]{Corollary}

\theoremstyle{definition}
\newtheorem{definition}[theorem]{Definition}

\theoremstyle{remark}

\numberwithin{equation}{section}
\def\fnote#1{\footnote}

\def\ignora#1{}
\def\n3#1{\left\vert  \! \left\vert \! \left\vert \, #1 \, \right\vert \!
  \right\vert \! \right\vert }

\newcommand{\pten}{\ensuremath{\widehat{\otimes}_\pi}}

\begin{document}

\title{A characterisation of the Daugavet property in spaces of vector-valued Lipschitz functions }

\author{Rub\'en Medina}
\address[R. Medina]{Universidad de Granada, Facultad de Ciencias. Departamento de An\'alisis Matem\'atico, 18071-Granada (Spain); and Czech Technical University in
Prague, Faculty of Electrical Engineering. Department of Mathematics, Technick\'a 2, 166 27 Praha 6 (Czech Republic)
	\newline
 \href{https://orcid.org/0000-0002-4925-0057}{ORCID: \texttt{0000-0002-4925-0057} }
	}
\email{\texttt{rubenmedina@ugr.es}}

\author{Abraham Rueda Zoca}
\address[A. Rueda Zoca]{Universidad de Granada, Facultad de Ciencias.
Departamento de An\'{a}lisis Matem\'{a}tico, 18071-Granada
(Spain)
	\newline
	\href{https://orcid.org/0000-0003-0718-1353}{ORCID: \texttt{0000-0003-0718-1353} }}
\email{\texttt{abrahamrueda@ugr.es}}
\urladdr{\url{https://arzenglish.wordpress.com}}



\maketitle

%\markboth{Rub\'en Medina and Abraham Rueda Zoca}{Daugavet property in tensor products of Lipschitz-free spaces}

\begin{abstract}
Let $M$ be a metric space and $X$ be a Banach space. In this paper we address several questions about the structure of $\mathcal F(M)\pten X$ and $\Lip(M,X)$. Our results are the following:

\begin{enumerate}
    \item We prove that if $M$ is a length metric space then $\Lip(M,X)$ has the Daugavet property. As a consequence, if $M$ is length we obtain that $\mathcal F(M)\pten X$ has the Daugavet property. This gives an affirmative answer to \cite[Question 1]{gpr18} (also asked in \cite[Remark 3.8]{lr2020}).
    \item We prove that if $M$ is a non-uniformly discrete metric space or an unbounded metric space then the norm of $\mathcal F(M)\pten X$ is octahedral, which solves \cite[Question 3.2 (1)]{blr18}.
    \item We characterise all the Banach spaces $X$ such that $L(X,Y)$ is octahedral for every Banach space $Y$, which solves a question by Johann Langemets.
\end{enumerate}

\end{abstract}

\section{Introduction}

In this note, we will focus mainly on the study of Lipschitz mappings from a metric space $M$ to a Banach space $X$, that is, we will focus on the space $\Lip(M,X)$. In particular, we are mainly interested in the Daugavet property on these spaces since it is known that $\Lip(M)$ enjoys it if, and only if, the metric space $M$ is length. The natural question that arises in this setting is therefore whether the same metric characterisation works in the vector valued case, that is, whether the space $\Lip(M,X)$ being Daugavet is equivalent to $M$ being length regardless which $X$ is used as codomain. This problem has been around since the latter characterisation of Lipschitz spaces was found in \cite[Theorem 3.5]{gpr18} and it has been explicitly asked in \cite[Question 1]{gpr18} and \cite[Remark 3.8]{lr2020}. We give here a complete positive solution.

\subsection{Motivation and background}

A Banach space $X$ is said to have the Daugavet property if every rank-one operator $T:X\longrightarrow X$ satisfies the equality
\begin{equation}\label{ecuadauga}
\Vert T+I\Vert=1+\Vert T\Vert,
\end{equation}
where $I$ denotes the identity operator. The previous equality is known as \emph{Daugavet equation} because I.~Daugavet proved in \cite{dau} that every compact operator on $\mathcal C([0,1])$ satisfies (\ref{ecuadauga}). Since then, many examples of Banach spaces enjoying the Daugavet property have appeared. E.g. $\mathcal C(K)$ for a perfect compact Hausdorff space $K$; $L_1(\mu)$ and $L_\infty(\mu)$ for a non-atomic measure $\mu$; or preduals of Banach spaces with the Daugavet property (see \cite{kkw,kssw01,wer} and references therein for a detailed treatment of the Daugavet property).

The question of determining when the space of Lipschitz functions $\Lip(M)$ has the Daugavet property has attracted the attention of many researchers since in \cite[Section 6, Question 1]{wer} D. Werner asked whether $\Lip([0,1]^2)$ has the Daugavet property. As a consequence of succesive works \cite{ikw,gpr18,am19} the following characterisation was obtained: Given a complete metric space $M$, then $\Lip(M)$ has the Daugavet property if, and only if, $M$ is length if, and only if, $\mathcal F(M)$ has the Daugavet property, where $\mathcal F(M)$ stands for the Lipschitz-free space over $M$ (also known as transportation cost space). All the above are equivalent to the absence of strongly exposed points of the unit ball of $\mathcal F(M)$. 

In the above line, the first result appeared in \cite[Proposition 3.11]{gpr18} where, as a main consequence, it is obtained that $\Lip(M,H)$ has the Daugavet property when $H$ is a Hilbert space. This result is obtained as a consequence of a well known result of extension of Lipschitz functions between Hilbert spaces due to Kirszbraun (see \cite[Theorem 1.12]{bl00}). More general results were obtained in \cite{rueda22}, where it was proved that $\Lip(M,X)$ has the Daugavet property for every Banach space $X$ if $M$ is a convex subset of a Hilbert space or if $M$ is a Banach space whose dual is isometrically $L_1(\mu)$. In all of the above cases, it is heavily used that given any finite subset $N$ of $M$ and any Lipschitz function $\varphi: N\longrightarrow M$, there exists an extension $\phi:M\longrightarrow M$ which almost preserves the Lipschitz constant. That is, a (demanding) property of extension of Lipschitz functions is required as a tool for making perturbations of given Lipschitz functions $f:M\longrightarrow X^*$ almost preserving their Lipschitz constant.

A feature which explains the difficulty of the aforementioned problem arises when one considers dual Banach targets. In such case, $\Lip(M,X^*)=L(\mathcal F(M),X^*)=(\mathcal F(M)\pten X)^*$, where $\mathcal F(M)\pten X$ stands for the projective tensor product of $\mathcal F(M)$ and $X$. Taking into account the immediate fact that the Daugavet property passes from a Banach space to its predual, the above mentioned open question is related to the question of how the Daugavet property is inherited by taking projective tensor products, which is a long standing open problem from \cite[Section 6]{wer}.

In this paper we explore new ideas for the perturbation of Lipschitz functions taking values on a Banach space in order to solve many open questions. Indeed, we obtain the following main result.

\begin{Theorem}\label{theo:mainDauga}
    If $M$ is a complete length metric space and $X$ is a Banach space then $\Lip(M,X)$ has the Daugavet property.
\end{Theorem}

As a corollary (Corollary \ref{cor:Daugavetvectorval}) we get that if $M$ is a complete length metric space and $X$ is a Banach space then $\mathcal F(M)\pten X$ has the Daugavet property (see the definition below). This solve \cite[Question 1]{gpr18} (also asked in \cite[Remark 3.8]{lr2020}) in complete generality.

One of the new tools behind the proof of the above mentioned Theorem \ref{theo:mainDauga} is Lemma \ref{loca0eident}, which informally says that given any Banach space $X$ there is an almost 1-Lipschitz perturbation of the identity mapping which vanishes at a ball centred at $0$ (we thank Richard Smith who pointed us out the existence of such functions). We also considered in Lemma \ref{prop:idenlocalycero} the ``dual construction'', that is, a function which is the identity on a neighbourhood of $0$ which vanishes out of a bounded subset of $X$.

%---------------------TERMINA PARTE DAUGAVET------------------

Next we continue using the ideas of perturbation of vector-valued Lipschitz functions to solve another open problem on the space $\mathcal F(M)\pten X$.

In order to put the problem into context \cite[Question 3.2 (1)]{blr18}, recall that a Banach space $X$ is said to be \textit{octahedral} if for every $x_1,\dots,x_n\in S_X$ and $\ep>0$ there is $y\in S_X$ such that $\|x_i-y\|\ge2-\ep$ for every $i=1,\dots,n$ (see Definition \ref{defi:octahedral} and the surrounding paragraphs for background). With the terminology of octahedral norms in mind, \cite[Question 3.2 (1)]{blr18} asks whether $\mathcal F(M)\pten X$ is octahedral whenever $M$ is a complete metric space which is not uniformly discrete and bounded. In Section \ref{section:octafree} we give a positive answer to this question.
\begin{Theorem}\label{octamain}
    Let $M$ be a metric space. If $M$ is not a bounded and uniformly discrete space then $\mathcal F(M)\pten X$ is octahedral for every Banach space $X$.
\end{Theorem}
Recall that a metric space is said to be uniformly discrete if 
$\inf\{d(x,y)\;;\;x,y\in M,\;x\neq y\}>0.$
% The proof is splited in two theorems: We first prove in Theorem \ref{theo:librediscrenouniform} the case in which $M$ is not uniformly discrete and we then prove in Theorem \ref{theo:octaunboununidis} the case when $M$ is not bounded.
Observe that this result constitutes a step forward for the open question \cite[Question 4.4]{llr}, where it is asked whether $X\pten Y$ is octahedral provided $X$ is octahedral.

We finish Section \ref{section:octafree} with Theorem \ref{theo:bidualfreeoh} where we prove that, when $M$ is a metric space with infinitely many cluster points, then we get the stronger condition that $(\mathcal F(M)\pten X)^{**}$ is octahedral.

We end the paper with Section \ref{section:universaloctaedral}, which is of independent interest, where we answer the following question: when does a Banach space $X$ satisfies that $L(X,Y)$ is octahedral for every Banach space $Y$? This question was  posed by J.~Langemets in the congress International Workshop on Operator Theory and its Applications (IWOTA 2022), Krakow, Poland (Sep 2022)).

% In Theorem \ref{theo:caraunivdomain} we prove, using recent techniques from \cite{rz2023} and local theory of Banach spaces, that a Banach space $X$ satisfies that $L(X,Y)$ is octahedral for every Banach space $Y$ if, and only if, the space $L(Y,X^*)$ is octahedral for every Banach space $Y$, which is in turn equivalent to the following condition on $X^*$: for every $\varepsilon>0$, for every finite dimensional subspace $Z$ of $X^*$ and for every $n\in\mathbb N$, there exists an operator $T:\ell_\infty^n\longrightarrow X^*$ with $\Vert T\Vert\leq 1$ and such that
% $$\Vert z+T(y)\Vert\geq (1-\varepsilon)(\Vert z\Vert+\Vert y\Vert)$$
% holds for every $y\in \ell_\infty^n$ and every $z\in Z$.

\subsection{Terminology}

We will consider only real Banach spaces. Given a Banach space $X$, $B_X$ (resp. $S_X$) stands for the closed unit ball (resp. the unit sphere).

We recall that
the {\it{projective tensor product of two Banach spaces $X$ and $Y$}}, denoted by
$X\pten Y$, is the completion of $X\otimes Y$
under the norm given by 
$$\Vert u\Vert:=\inf \left \{\sum_{i=1}^n \Vert x_i\Vert\Vert y_i\Vert\ :\ n\in\mathbb N, x_i\in X, y_i\in Y,  u=\sum_{i=1}^n x_i\otimes y_i \right\},$$ for every $u\in X\otimes Y$. 

We recall that the space $L(X, Y^*)$ is linearly isometric to the topological dual of $X\pten Y$ by the action $T(x\otimes y):=T(x)(y)$. As an immediate consequence of the above identification we get that $B_{X\pten Y}=\overline{\co}(B_X\otimes B_Y)$. We refer the reader to \cite{ryan} for background on tensor product spaces.

Also, since we are mainly interested in the case when one of the factors is a Lipschitz free space, let us introduce that notion from scratch. A pointed metric space $(M, d, 0)$
is a metric space $(M, d)$ with a selected distinguished point 0 in $M$, called
the base point. For a pointed metric space $M$, we consider the Banach
space $\Lip(M)$ formed by all real-valued Lipschitz functions that vanish at 0 endowed with the norm given by the least Lipschitz constant, that is,
$$\|f\|_{Lip}=\inf\{C>0\;:\;|f(x)-f(y)|\le Cd(x,y)\;,\;\forall x,y\in M\},$$
for every $f\in \Lip(M)$. The Banach space $\Lip(M)$ is a dual space and its
canonical predual is our sought space, often called the \textit{Lipschitz free space} associated to $M$, denoted by
$\mathcal{F}(M)$. It is given by $\mathcal{F}(M) = \overline{\text{span}}\{
\delta_p\}_{p\in M} \subseteq \Lip(M)^{**}$, where
$\delta_p: \Lip(M) \to \R$ is the linear functional defined by $\delta_p(f) = f(p)$ for all $f\in \Lip(M)$. Now, if $X$ is a Banach space we denote $\Lip(M,X)$ to the space of Lipschitz mappings from $M$ to $X$ vanishing at 0 endowed with the Lipschitz norm. We are interested in vector-valued Lipschitz mappings since the isometric identification $\Lip(M,X^*)=L(\mathcal F(M),X^*)$ leads a structure of dual Banach space and its canonical predual is our object of study $\mathcal{F}(M)\pten X$.

We must define now the metric spaces $M$ such that $\Lip(M)$ is Daugavet. These spaces are known as length spaces (see e.g. \cite{bbi}). Specifically, a metric space $(M,d)$ is a \textit{length} space whenever it is arc-connected and for every $x,y\in M$,
$$d(x,y)=\inf\{\text{l}(\alpha)\;:\;\alpha:[0,1]\to M,\;\alpha(0)=x,\;\alpha(1)=y\},$$
where $\text{l}(\alpha)$ stands for the length of the curve $\alpha$. In the proof of the equivalence between $M$ length and $\mathcal{F}(M)$ Daugavet the notion of spreadingly local metric spaces is used and we will need it, too. According to \cite{ikw}, a metric space is said to be \textit{spreadingly local} whenever
$$A_{\ep,f}=\{m\in M\;:\;\inf_{r>0}\|f\|_{Lip}>1-\ep\}$$
is infinite for every $\ep>0$ and $f\in\Lip(M)$. Following the proof of Proposition 2.3 in \cite{ikw} it is easy to see that every length space is spreadingly local, regardless it is complete or not.

In the setting of metric spaces, we are going to make use of the well-known Mcshane extension of Lipschitz functions. Given a metric space $M$, a subset $N$ of $M$ and $f\in \Lip(N)$ there is always an extension $g\in \Lip(M)$ of $f$ to $M$, often called the McShane extension, such that $\|g\|_{Lip}=\|f\|_{Lip}$. In general, there is no McShane extension for vector-valued Lipschitz mappings, which is the main problem that we face.

Let us define now the main geometric properties of Banach spaces we are going to work with throughout the entire note. A slice $S$ of the unit ball $B_X$ of a Banach space $X$ is a subset of $B_X$ of the form
$$S=S(B_X,f,\alpha)=\{z\in B_X\;:\;f(z)>1-\alpha\},$$
For some $f\in X^*$ and $\alpha>0$. If $X$ is a dual space and $f$ belongs to its predual we call the previous set a $w^*$-slice.
\begin{definition}\label{defi:Daugavet}
    Given $X$ a Banach space, we say that $X$ has the \textit{Daugavet property} (or simply that $X$ is Daugavet)  whenever for every $x\in S_X$, $\ep>0$ and $S$ slice of $B_X$ there is $y\in S$ such that $\|x-y\|\ge2-\ep$. If $X$ is a dual space and the latter holds at least for $w^*$-slices we say that $X$ has the $w^*$-Daugavet property (or simply $X$ is $w^*$-Daugavet).
    \end{definition}
Notice that this definition of the Daugavet property is equivalent to the one given at the beginning of the paper \cite[Lemmata 2.1 and 2.2]{kssw01}.

Now we will present two more definitions closely related to the Daugavet property.

\begin{definition}\label{defi:octahedral}
    Given $X$ a Banach space, we say that $X$ is \textit{octahedral} if for every $x_1,\dots,x_n\in S_X$ and $\ep>0$ there is $y\in S_X$ such that $\|x_i-y\|\ge2-\ep$ for every $i=1,\dots,n$.
\end{definition}

\begin{definition}\label{defi:SD2P}
    Given $X$ a Banach space, we say that $X$ enjoys the \textit{strong diameter 2 property} (SD2P for short) whenever every convex combination of slices has diameter 2.
\end{definition}
It is worth mentioning that by \cite[Theorem 2.1]{blr14}, for a Banach space $X$ it is equivalent being octahedral to its dual $X^*$ having the $w^*$-SD2P. Moreover, it follows that if $X$ is a Daugavet space then both $X$ and $X^*$ have octahedral norms \cite[Lemmata 2.8 and 2.12]{kssw01}. For background on diameter two properties, octahedral norms and their interrelations with the Daugavet property we refer the reader to \cite{langemetsthesis,pirkthesis,ruedathesis}.

% XXXXXXXXXXXXXXX ABRAHAM: HAY QUE INTRODUCIR
% 1) NOTACION DE TENSORES PROYECTIVOS
% 2) NOTACION DE LIPSCHITZ FREE
% 3) PROPIEDAD DE DAUGAVET, DIAMETROS DOS, OCTAEDRALIDAD, IMPLICACIONES, REFERENCIAS, OBJETOS NECESARIOS (SLICES, COMBINACIONES CONVEXAS DE SLICES)
% 4) AP,MAP, OPERADORES, OPERADORES DE RANGO FINITO. 

% 6) ESPACIOS METRICOS. LENGTH. EPSILON PUNTOS DE FUNCIONES LIPSCHITZ.
% 7) MCSHANE. COMENTARIOS DE QUE EL VECTORIAL ES FALSO. COMPLEMENTACION LOCAL, LUST.
% 8) RETRACTOS. RELACION CON EXTENSION DE FUNCIONES LIPSCHITZIANAS.

\section{Lipschitz functions on Banach spaces}

Let us start with the following straigthforward lemmata which will shorten future arguments.

\begin{lemma}\label{l1}
    Let $X$ be any Banach space, $\lambda\geq0$ and $f:X\to X$. If $f$ is $\lambda$-Lipschitz when restricted to segments then $f$ is $\lambda$-Lipschitz.
\end{lemma}

\begin{lemma}\label{l2}
    Let $X$ be any Banach space, $a<b<c$ in $\R$ and $f:[a,c]\to X$. If $\restr{f}{[a,b]}$ and $\restr{f}{[b,c]}$ are $\lambda$-Lipschitz for some $\lambda\geq0$ then $f$ is $\lambda$-Lipschitz.
\end{lemma}
\begin{proof}
    Let $t,s\in[a,c]$ distinct. We may assume that $t\in[a,b]$ and $s\in[b,c]$. Then,
    $$\|f(t)-f(s)\|\leq\|f(t)-f(b)\|+\|f(b)-f(s)\|\leq\lambda(|t-b|+|b-s|)=\lambda|t-s|.$$\end{proof}

\begin{lemma}\label{lemma:segments}
Let $X$ be a Banach space, let $0<r<R$ and consider a function $f:X\longrightarrow X$. Assume that there are $a,b,c>0$ such that $\sup\limits_{x,y\in B(0,r),x\neq y} \frac{\Vert f(x)-f(y)\Vert}{\Vert x-y\Vert}\leq a$, $\sup\limits_{x,y\in C(0,r,R),x\neq y} \frac{\Vert f(x)-f(y)\Vert}{\Vert x-y\Vert}\leq b$ and \break $\sup\limits_{x,y\in X\setminus B(0,R),x\neq y} \frac{\Vert f(x)-f(y)\Vert}{\Vert x-y\Vert}\leq c$. Then $f$ is Lipschitz and $\Vert f\Vert\leq \max\{a,b,c\}$.
\end{lemma}

\begin{proof}
A combination of Lemma \ref{l1} and Lemma \ref{l2} does the trick.\end{proof}
% Set $M:=\max\{a,b,c\}$. Let us prove that $f$ is Lipschitz. In order to do so, take $x,y\in X$ with $x\neq y$, and let us prove that $\Vert f(x)-f(y)\Vert\leq M \Vert x-y\Vert$. If both $x$ and $y$ belong simultaneously to either $B(0,r)$ or $C(0,r,R)$ or $X\setminus B(0,R)$ then the conclussion follows from the assumptions. Because of that we will assume a different situation. 

% Let us prove, for instance, the case when $x\in B(0,r)$ and $y\in C(0,r,R)$, the other cases will run similarly. Consider the vector segment $[x,y]\subseteq X$. An easy continuity argument with the function $t\longmapsto \Vert tx+(1-t)(y)\Vert, t\in [0,1]$ implies the existence of $z\in [x,y]$ such that $\Vert z\Vert=r$ (in other words $z\in B(0,r)\cap C(0,r,R)$. Since $z\in B(0,r)$ we get $\Vert f(x)-f(z)\Vert\leq a\Vert x-z\Vert$, whereas the condition $y\in C(0,r,R)$ implies $\Vert f(z)-f(y)\Vert\leq b\Vert z-y\Vert$. Now
% \[
% \begin{split}\Vert f(x)-f(y)\Vert\leq \Vert f(x)-f(z)\Vert+\Vert f(z)-f(y)\Vert& \leq a\Vert x-z\Vert+b\Vert z-y\Vert\\
% & \leq M(\Vert x-z\Vert+\Vert z-y\Vert).
% \end{split}\]
% Since $z\in [x,y]$ we infer that $\Vert x-y\Vert=\Vert x-z\Vert+\Vert z-y\Vert$. Consequently we derive that $\Vert f(x)-f(y)\Vert\leq M\Vert x-y\Vert$, as desired.

% The cases when $x\in C(0,r,R), y\in X\setminus B(x,R)$ or $x\in B(0,r), y\in X\setminus B(x,R)$ run similarly and finish the proof.

In the following result we will construct a Lipschitz function on a Banach space which coincides with the identity map out of a ball centred at the origin and that it is flat at a smaller ball contained. This function, which to our knownledge appeared in the thesis of Filip Talimdjiosk, was provided to the authors by his thesis advisor R. Smith. The authors are deeply grateful to R. Smith for providing them the following proposition and its proof.

\begin{proposition}\label{loca0eident}
Let $X$ be a Banach space and let $0<a<b$. Then the function $f:X\longrightarrow X$ defined by
$$f(x):=\left\{\begin{array}{cc}
   0  &  \mbox{if }\Vert x\Vert\leq a,\\
   \frac{b}{b-a}\left(1-\frac{a}{\Vert x\Vert} \right) x  & 
\mbox{if }a\leq \|x\|\leq b,\\
x & \mbox{if }b\leq \|x\|,
\end{array} \right.$$
is Lipschitz and $\Vert f\Vert\leq \frac{b}{b-a}$.

In particular, it follows that for every $x_0\in X$ and every pair of positive numbers $R, \varepsilon$ there exist $\delta>0$ and a Lipschitz-mapping $\psi:X\longrightarrow X$ such that $\psi(x)=x$ holds for every $x\in X\setminus B(x_0,R)$, $\Vert \psi\Vert\leq 1+\varepsilon$ and $\psi(z)=x_0$ holds for every $z\in B(x_0,\delta)$. 
\end{proposition}

\begin{proof}
Let $x,y\in X$ with $x\neq y$. In virtue of Lemma \ref{lemma:segments}, and since $f$ is clearly Lipschitz at $B(a,r)$ and on $X\setminus B(0,b)$, let us assume that $x,y\in C(0,a,b)$. We assume with no loss of generality that $\Vert x\Vert\leq \Vert y\Vert$. Now we compute $\Vert f(x)-f(y)\Vert$.
\[\begin{split}
\Vert f(x)-f(y)\Vert& = \frac{b}{b-a}\left\Vert y-x+\frac{a}{\Vert x\Vert}x-\frac{a}{\Vert y\Vert}y\right\Vert\\
& =\frac{b}{b-a}\left\Vert  y-x+\frac{a}{\Vert x\Vert}x-\frac{a}{\Vert y\Vert}y+\frac{a}{\Vert x\Vert}y-\frac{a}{\Vert x\Vert}y \right\Vert\\
& \leq \frac{b}{b-a}\left(\left(1-\frac{a}{\Vert x\Vert}\right)\Vert x-y\Vert+\left\vert \frac{a}{\Vert x\Vert}-\frac{a}{\Vert y\Vert}\right\vert \Vert y\Vert \right)\\
& = \frac{b}{b-a}\left(\left(1-\frac{a}{\Vert x\Vert}\right)\Vert x-y\Vert+\frac{a\vert \Vert x\Vert-\Vert y\Vert\vert}{\Vert x\Vert \Vert y\Vert} \Vert y\Vert \right)\\
& \leq \frac{b}{b-a}\left(\left(1-\frac{a}{\Vert x\Vert}\right)\Vert x-y\Vert+\frac{a}{\Vert x\Vert}\Vert x-y\Vert \right)\\
& =\frac{b}{b-a}\Vert x-y\Vert,
\end{split}\]
as desired.

For the second part of the theorem, given $x_0\in X$ and $R,\varepsilon>0$, take $\delta>0$ such that $\frac{R}{R-\delta}<1+\varepsilon$, consider the function above $\varphi:X\longrightarrow X$, then the function $\psi(z):=x_0+\varphi(x-x_0)$ does the trick.\end{proof}

In the following result we will construct a function whose behaviour is the opposite, i.e., a function which is the identity in a neighbourhood of $0$ and that is constantly $0$ out of a certain ball centred at $0$.

\begin{proposition}\label{prop:idenlocalycero}
Let $X$ be a Banach space and $0<a<b$. The function $f:X\longrightarrow X$ defined by
$$f(x):=\left\{\begin{array}{cc}
    x &  \mbox{if }x\in B(0,a),\\
    \frac{b-\Vert x\Vert}{b-a}x & \mbox{if }x\in C(0,a,b),\\
    0 & \mbox{if } x\in X\setminus B(0,R),
\end{array} \right.$$
is Lipschitz and $\Vert f\Vert\leq \frac{b}{b-a}$.
\end{proposition}

\begin{proof}
Again by an application of Lemma \ref{lemma:segments} and because $f$ is clearly Lipschitz on $B(0,a)$ and on $X\setminus B(0,b)$, it is enough to prove that, given $x,y\in C(0,a,b)$ with $x\neq y$, we get $\Vert f(x)-f(y)\Vert\leq \frac{b}{b-a}\Vert x-y\Vert$. Take such $x,y\in C(0,a,b)$ and assume with no loss of generality that $\Vert x\Vert\leq \Vert y\Vert$. Then
\[
\begin{split}
\Vert f(x)-f(y)\Vert& =\frac{1}{b-a}\left\Vert (b-\Vert x\Vert) x-(b-\Vert y\Vert) y \right\Vert\\
& =\frac{1}{b-a}\left\Vert(b-\Vert y\Vert)x+(\Vert y\Vert-\Vert x\Vert)x-(b-\Vert y\Vert)y \right\Vert\\
& \leq\frac{1}{b-a}((b-\Vert y\Vert)\Vert x-y\Vert +\vert \Vert x\Vert-\Vert y\Vert\vert \Vert x\Vert )\\
& \leq \frac{1}{b-a}((b-\Vert x\Vert)\Vert x-y\Vert+\Vert x-y\Vert \Vert x\Vert)\\
& =\frac{b}{b-a}\Vert x-y\Vert.
\end{split}
\]
\end{proof}

The following lemma shows that the set of Lipschitz functions which is injective on a given sequence is norm-dense. We establish and prove the result because we will use it several times throughout the text.

\begin{lemma}\label{lemma:separainfisequ}
Let $M$ be a metric space and $X$ be a Banach space. Consider a sequence of pairwise-disjoint balls $\{B(x_n,r_n)\}$ in $M$. Then for every Lipschitz function $F:M\longrightarrow X$ and $\varepsilon>0$ there exists a Lipschitz function $G:M\longrightarrow X$ with the following properties:
\begin{enumerate}
    \item $\Vert F-G\Vert<\varepsilon$ and,
    \item $G(x_n)\neq G(x_m)$ holds for every $n\neq m$.
\end{enumerate}
\end{lemma}

\begin{proof}
Given $n\in\mathbb N$ we can take, by McShane extension theorem, a Lispchitz function $\varphi_n:M\longrightarrow \mathbb R$ such that $\varphi_n(x_n)\neq 0$ for every $n\in\mathbb N$, $\Vert \varphi_n\Vert=1$ and $\varphi_n=0$ on $M\setminus B(x_n,r_n)$. Let $F$ and $\varepsilon$ as in the hypothesis, and consider a sequence $\{\varepsilon_n\}$ of strictly positive numbers such that $\sum_{n=1}^\infty \varepsilon_n<\varepsilon$. Take also $x\in S_X$. 

We will construct by induction a sequence $\{\delta_n\}$ of positive numbers such that, for every $n\in\mathbb N$, the following conditions hold:
\begin{enumerate}
    \item $\delta_n\leq \varepsilon_n$ holds for every $n\in\mathbb N$ and;
    \item $F(x_i)+\delta_i \varphi(x_i)x\neq F(x_n)+\delta_n \varphi(x_n)x$ holds for $1\leq i\leq n-1$.
\end{enumerate}
For $n=1$ take $\delta_1=\varepsilon_1$. Now assume $\delta_1\ldots, \delta_n$ have been constructed and let us construct $\delta_{n+1}$. In order to do so, observe that the set
$$\{F(x_{n+1})+\delta \varphi_{n+1}(x_{n+1})x: 0<\delta<\varepsilon_{n+1}\}$$
is an infinite set since $x$ is a non-zero vector. Since the set
$$\{F(x_i)+\delta_i \varphi_i(x_i)x: 1\leq i\leq n\}$$
is a finite set we can find $0<\delta_{n+1}<\varepsilon_{n+1}$ such that 
$$F(x_{n+1})+\delta \varphi_{n+1}(x_{n+1})x\notin \{F(x_i)+\delta_i \varphi_i(x_i)x: 1\leq i\leq n\}.$$
Now $G=F+\sum_{n=1}^\infty \delta_n \varphi_n\otimes x$ satisfies our requirements. To begin with, the inequality $\Vert F-G\Vert\leq \sum_{n=1}^\infty \delta_n \Vert \varphi_n\Vert \Vert x\Vert\leq \sum_{n=1}^\infty \varepsilon_n<\varepsilon$ holds. On the other hand, the construction of $\delta_n$ together with the fact that $\supp(\varphi_n)\subseteq B(x_n,r_n)$ implies that $G(x_n)=F(x_n)+\delta_n \varphi_n(x_n)x$, from where the proof follows.\end{proof}

Now we introduce the following result, whose proof is easy, and which will save us a lot of notation in the following. The proof is straightforward, but let us include it for the sake of completeness.

\begin{lemma}\label{lemma:compuortolip}
Let $M$ be a complete metric space and $X$ be a Banach space. Let $f,g:M\longrightarrow X$ be two  Lipschitz functions. Assume that there exists $m\in M$ and $0<\delta<R$ so that
\begin{enumerate}
\item $g$ is constant on $B(m,R)$.
\item $f(x)=f(m)$ holds for every $x\in M\setminus B(m,\delta)$.
\end{enumerate}
Then $\Vert f+g\Vert\leq \max\{\Vert f\Vert, \Vert g\Vert\} (1+\frac{2\delta}{R-\delta}).$
\end{lemma}

\begin{proof} Call $M:=\max\{\Vert f\Vert, \Vert g\Vert\}$. Let $x, y\in M$ with $x\neq y$. Let us estimate
$$A:=\frac{\Vert (f(x)+g(x))-(f(y)+g(y))\Vert}{d(x,y)}=\frac{\Vert f(x)-f(y)+g(x)-g(y)\Vert}{d(x,y)}.$$
Let us observe that if $f(x)-f(y)=0$ or $g(x)-g(y)=0$ then $A\leq M$.

The unique possibility for the previous condition not to hold is that, up to relabeling, $x\notin B(m,R)$ and $y\in B(m,\delta)$. In that case $f(x)=f(m)$ and $g(y)=g(m)$. Consequently
$$A\leq \frac{\Vert f(y)-f(m)\Vert+\Vert g(x)-g(m)\Vert}{d(x,y)}\leq M \frac{d(y,m)+d(x,m)}{d(x,y)}.$$
Since $d(x,m)\leq d(x,y)+d(y,m)$ the above inequality yields
$$A\leq M \frac{d(x,y)+2d(y,m)}{d(x,y)}=M\left(1+\frac{2d(y,m)}{d(x,y)}\right).$$
Now using $d(x,y)\geq d(x,m)-d(y,m)\geq R-\delta$ we get
$$A\leq M\left(1+\frac{2\delta}{R-\delta}\right),$$
as desired.
\end{proof}

Let us end by recalling the following criterion of weakly null sequences on $\Lip(M,X)$ from \cite{ccgmr19}, which will be used several times throughout the text and whose proof can be found in \cite[Lemma 1.5]{ccgmr19}.

\begin{lemma}\label{lemmaweaknull}
Let $M$ be a pointed metric space, let $X$ be a Banach space, and let $\{f_n\}$ be a sequence of functions in the unit ball of $\Lip(M,X)$. For each $n\in \mathbb N$, we write $U_n:=\{x\in M \colon f_n(x)\neq 0\}$ for the support of $f_n$. If $U_n\cap U_m = \emptyset$ for every $n\neq m$, then the sequence $\{f_n\}$ is weakly null.
\end{lemma}

\section{Daugavet property in $\Lip(M,X)$}
 In this section we will make use of the previous results to give the prove of Theorem \ref{theo:mainDauga}.

\begin{proof}[Proof of Theorem \ref{theo:mainDauga}]
Let $f,g\in S_{\Lip(M,X)}$, and, in order to prove that $\Lip(M,X)$ has the Daugavet property, let us prove that for every $\varepsilon>0$ there exists a sequence $(f_n)\subseteq (1+\varepsilon)B_{\Lip(M,X)}$ such that $(f_n)\longrightarrow g$ weakly and $\Vert f+f_n\Vert\geq 2-\varepsilon$ holds for every $n\in\mathbb N$. This is enough thanks to \cite[Theorem 2.1. (5)]{rueda22}. 

Since $\Vert f\Vert_{Lip}=1$ we can find $y^*\in S_{X^*}$ such that $y^*\circ f:M\longrightarrow \mathbb R$ given by $y^*\circ f(m):=y^*(f(m))$ satisfies $\Vert y^*\circ f\Vert_{Lip}>1-\frac{\varepsilon}{2}$.

Since $M$ is length, it is spreadingly local (see the terminology section), that is, the set
$$A=\left\{m\in M\;:\;\inf_{r>0}\|y^*\circ f\restricted{B(m,r)}\|_{Lip}>1-\varepsilon/4\right\}$$
is infinite. Hence we can take a sequence of pairwise distinct points $(x_n)\subseteq A$. Up to an application of Lemma \ref{lemma:separainfisequ} we can assume, up to a norm-perturbation argument, that $g(x_n)\neq g(x_m)$ if $n\neq m$. Hence we can find, for every $n\in\mathbb N$, an element $\alpha_n>0$ such that $\{B(g(x_n),\alpha_n): n\in\mathbb N\}$ is pairwise disjoint (observe that since $\Vert g\Vert\leq 1$ it is clear that $\{B(x_n,\alpha_n)\}$ is also pairwise disjoint).

Consider $0<\beta_n<\alpha_n$ for every $n\in\mathbb N$ such that $\frac{\alpha_n}{\alpha_n-\beta_n}\rightarrow 1$ and consider, in virtue of Proposition \ref{loca0eident}, a Lipschitz function $\varphi_n:X\longrightarrow X$ such that $\Vert \varphi_n\Vert\leq \frac{\alpha_n}{\alpha_n-\beta_n}$ holds for every $n\in\mathbb N$, that $\varphi_n(x)=x$ holds for every $x\in X\setminus B(g(x_n),\alpha_n)$ and $\varphi_n(x)=g(x_n)$ holds for every $x\in B(g(x_n),\beta_n)$. 

Now for every $n\in\mathbb N$ write $h_n:=\varphi_n\circ g: M\longrightarrow X$. It is immediate that $\Vert h_n\Vert\rightarrow 1$ (since $\Vert h_n\Vert\leq \Vert \varphi_n\Vert \Vert g\Vert\leq \frac{\alpha_n}{\alpha_n-\beta_n})$. 

Now we claim that $(h_n-g)$ is a sequence of mappings with  pariwise disjoint support. Indeed, it is immediate that, given $n\in\mathbb N$, it follows that $h_n(x)-g(x)\neq 0$ implies $g(x)\in B(g(x_n),\alpha_n)$ which, in other words, mean that $\supp(h_n-g)\subseteq g^{-1}(B(g(x_n),\alpha_n))$ for every $n\in\mathbb N$. The fact that $\supp(h_n-g)$ is pairwise disjoint is immediate now since $B(g(x_n),\alpha_n)$ is pairwise disjoint. Consequently $h_n-g$ is weakly null in virtue of Lemma \ref{lemmaweaknull} or, equivalently, $(h_n)\rightarrow g$ weakly. 

On the other hand observe that $h_n=\varphi_n\circ g$ takes the value $g(x_n)$ at $B(x_n,\beta_n)$. Indeed, given $z\in B(x_n,\beta_n)$ it follows that $\Vert g(z)-g(x_n)\Vert\leq \Vert g\Vert d(z,x_n)<\beta_n$, which implies $\varphi_n(g(z))=g(x_n)$ by the very definition of $\varphi_n$.

Now consider a sequence $(r_n)$ of strictly positive numbers such that $\frac{r_n}{\beta_n-2r_n}\rightarrow 0$. 

Since $\{x_n:n\in\mathbb N\}\subseteq A$ we can find $y_n\in M$ with $0<d(x_n,y_n)<r_n$ such that 
$$\frac{y^*(f(y_n))-y^*(f(x_n))}{d(y_n,x_n)}>1-\frac{\varepsilon}{4}.$$
Now define a function $s_n:M\longrightarrow \mathbb R$ with $\Vert s_n\Vert\leq 1$, $s_n(z)=0$ if $z\in M\setminus B(x_n,2r_n)$, $s_n(x_n)=0$ and $s_n(y_n):=d(x_n,y_n)$. This function may be constructed as the McShane extension to $M$ of the function $$\widetilde s_n:M\setminus B(x_n,2r_n)\cup\{x_n,y_n\}\to\R$$
defined as it was previously stated which is clearly $1$-Lipschitz.

Now since $\Vert y^*\Vert=1$ find $y\in S_X$ such that $y^*(y)>1-\frac{\varepsilon}{4}$. Consider $S_n:=s_n\otimes y:M\longrightarrow X$. Observe that $\Vert S_n\Vert\leq 1$. Moreover $S_n$ is clearly a sequence of pairwise disjoint Lipschitz functions, since the support of $S_n$ is contained in $B(x_n,r_n)\subseteq B(x_n,\alpha_n)$. Consequently by Lemma \ref{lemmaweaknull} we get $(S_n)\rightarrow 0$ weakly.

Now define $g_n:=h_n+S_n$, and we claim that the sequence $\{g_n\}$ does the trick. On the one hand, the convergence conditions on $h_n$ and $S_n$ imply that $g_n\rightarrow g$ weakly. On the other hand, given $n\in\mathbb N$, an appeal to Lemma \ref{lemma:compuortolip} for $m=x_n$, $R=\beta_n$ and $\delta=2r_n$ implies 
$$\Vert g_n\Vert\leq \frac{\alpha_n}{\alpha_n-\beta_n}(1+\frac{4 r_n}{\beta_n-2r_n})\rightarrow 1.$$
Consequently there exists $m\in\mathbb N$ such that $n\geq m$ implies $g_n\in (1+\varepsilon)B_{\Lip(M,X)}$. Finally, given $n\in\mathbb N$, we have
\[
\begin{split}
\Vert f+g_n\Vert& \geq y^*\left( \frac{(f+g_n)(y_n)-(f+g_n)(x_n)}{d(y_n,x_n)}\right)\\
& =\frac{y^*(f(y_n))-y^*(f(x_n))}{d(y_n,x_n)}+\frac{y^*(S_n)(y_n)-y^*(S_n)(x_n)}{d(y_n,x_n)}\\
& >1-\frac{\varepsilon}{4}+\frac{s_n(y_n)-s_n(x_n)}{d(y_n,x_n)}y^*(y)>1-\frac{\varepsilon}{4}+1-\frac{\varepsilon}{4}>2-\varepsilon.
\end{split}
\]
This finishes the proof.\end{proof}

As a particular case of the above theorem we have the following corollary.

\begin{corollary}\label{cor:Daugavetvectorval}
Let $M$ be a metric space and let $X$ be a non-zero Banach space. If $M$ is length then $(\mathcal F(M)\pten X)^*$ has the Daugavet property. In particular, $\mathcal F(M)\pten X$ has the Daugavet property.
\end{corollary}

\section{Octahedrality of $\mathcal F(M)\pten X$}\label{section:octafree}

In this section we will focus on the proof of Theorem \ref{octamain}. The proof is splited in two theorems: We first prove in Theorem \ref{theo:librediscrenouniform} the case in which $M$ is not uniformly discrete and we then prove in Theorem \ref{theo:octaunboununidis} the case when $M$ is not bounded.

Given a metric space $M$, for a given point $x\in M$ define $d_x=0$ if $x$ is a cluster point, otherwise define
$$d_x:=\sup\{r>0: B(x,r)=\{x\}\}.$$
Now we have the following result.

\begin{theorem}\label{theo:librediscrenouniform}
Let $M$ be a metric space which is not uniformly discrete and let $X$ be a Banach space. Then the norm of $\mathcal F(M)\pten X$ is octahedral.    
\end{theorem}

\begin{proof}
As before let $\mu_1,\ldots, \mu_q\in \mathcal F(M)\pten X$ and $\varepsilon>0$, and let us find $\mu\in S_{\mathcal F(M)\pten X}$ such that $\Vert \mu_i+\mu\Vert\geq \frac{(1-\varepsilon)^2}{1+\varepsilon}(\Vert \mu_i\Vert+1)$.

By a density argument we can assume that, for every $1\leq i\leq q$, we can write $\mu_i:=\sum_{j=1}^{n_i}\lambda_{ij}\delta_{m_{ij}}\otimes x_{ij}$, for certain $n_i\in\mathbb N$, $\lambda_{ij}\in\mathbb R\setminus\{0\}, m_{ij}\in M\setminus \{0\}$ and $x_{ij}\in X\setminus\{0\}$. 

By Hahn-Banach theorem we can find, for every $i\in\{1,\ldots, q\}$, an element $f_i\in (\mathcal F(M)\pten X)^*=\Lip(M,X^*)$ such that $\Vert f_i\Vert< 1$ and $f_i(\mu_i)>(1-\varepsilon)\Vert \mu_i\Vert$.

Since $M$ is discrete but not uniformly discrete we can find a sequence $\{x_n\}\subseteq M$ such that $d_{x_n}\rightarrow 0$.

Given $1\leq i\leq q$ we have two possibilities, depending on whether or not $f_i(x_{n_k})$ has a Cauchy subsequence:
\begin{enumerate}
\item There exists a subsequence $\{x_{n_k}\}$ of $\{x_n\}$ such that $f_i(x_{n_k})$ is norm convergent.
\item There exists $\delta_i>0$ such that $\Vert f_i(x_{n_k})-f_i(x_{n_p})\Vert\geq \delta_i$ if $k\neq p$.
\end{enumerate}
Thus we will assume that, up to relabeling the sequence $\{x_n\}$, there exists $\delta>0$ with the following property: given $i\in \{1,\ldots, q\}$ then either $f_i(x_n)$ is norm-convergent or $\Vert f_i(x_n)-f_i(x_m)\Vert\geq \delta$ if $n\neq m$. 

The above allows us, in any case of $i$, to assume up to a perturbation argument calling Lemma \ref{lemma:separainfisequ} that there exists some $R>0$ such that $\Vert f_i(x_n)-f_i(m_{ij})\Vert\geq R$ holds for every $n\in\mathbb N$, every $1\leq i\leq q$ and every $1\leq j\leq n_i$.

Take $r>0$ such that $\frac{R}{R-r}\Vert f_i\Vert<1$ holds for every $1\leq i\leq q$ and define $\varphi_n^i: X^*\longrightarrow X^*$ a $\frac{R}{R-r}$-Lipschitz function such that $\varphi_n^i(x^*)=x^*$ holds for every $x^*\in X^*\setminus B(f_i(x_n),R)$ and $\varphi_n^i(z)=f_i(x_n)$ if $z\in B(f_i(x_n),r)$. If we define $g_n^i:=\varphi_n^i\circ f_i:M\longrightarrow X^*$ we get that $\Vert g_n^i\Vert<1$, that $g_n^i(m_{ij})=f_i(m_{ij})$ (in particular $g_n^i(\mu_i)>(1-\varepsilon)\Vert \mu_i\Vert$) and moreover $g_n^i(z)=f_i(x_n)$ for every $z\in B(x_n,r)$.

Now select $\alpha>0$ small enough so that $\frac{2\alpha}{r-\alpha}<\varepsilon$ and find $n$ large enough so that $d_{x_n}<\frac{\alpha}{2}$. Now we consider $y_n\neq x_n$ with $d(x_n,y_n)<\frac{\alpha}{2}$ and define, using the McShane extension theorem and functions of the form $s\otimes v^*$ for $s:M\longrightarrow \mathbb R$, a Lipschitz function $S_n:M\longrightarrow X^*$ so that $S_n=0$ on $M\setminus B(x_n,\alpha)\cup\{x_n\}$, $S_n(y_n)=d(x_n,y_n)v^*$ for some $v^*\in S_{X^*}$ and $\Vert S_n\Vert=1$.

Then, consider $g_i:=g_n^i+S_n$, whose norm is smaller than $1+\frac{2\alpha_n}{r-\alpha_n}<1+\varepsilon$ by applying Lemma \ref{lemma:compuortolip} for $m=x_n$.

% Take some $u,v\in B(x_n,\alpha)$ with $u\neq v$ and $x\in S_X$ such that $\frac{S_n(u)(x)-S_n(v)(x)}{d(u,v)}>(1-\varepsilon)^2$ as in the final part of the proof of Theorem \ref{theo:normaohfreetensorcluster}. 
Finally, taking $v\in S_X$ with $v^*(v)>1-\varepsilon$ and putting $\mu=\frac{\delta_{y_n}-\delta_{x_n}}{d(y_n,x_n)}\otimes v$ we get that
\[
\begin{split}
(1+\varepsilon)\Vert \mu_i+\mu\Vert\geq g_i(\mu_i+\mu)& =g_i(\mu_i)+g_i(\mu)\\&=g_n^i(\mu_i)+S_n(\mu)\\
& >(1-\varepsilon)\Vert \mu_i\Vert+(1-\varepsilon)^2\\
& \geq (1-\varepsilon)^2(\Vert \mu_i\Vert+1),
\end{split}
\]
which implies $\Vert \mu_i+\mu\Vert\geq \frac{(1-\varepsilon)^2}{1+\varepsilon}(\Vert \mu_i\Vert+1)$, and the proof is finished.

% Now similar estimates to that of the above result allows to prove that $g_i(\mu_i+\mu)\geq (1-\varepsilon)^2(\Vert \mu_i\Vert+1)$, so $\Vert \mu_i+\mu\Vert\geq \frac{(1-\varepsilon)^2}{1+\varepsilon}(\Vert \mu_i\Vert+1)$, and the proof is finished.
\end{proof}

Now, we move on to the unbounded case. For that, we first need the following auxiliary result.

\begin{lemma}\label{lemma:densiunidiscre0}
Let $M$ be any metric space and let $X$ be a Banach space. Let $F\subseteq M$ be a finite set and $f:M\longrightarrow X$ be a Lipschitz function with $\Vert f\Vert<1$. Then there exists $g:M\longrightarrow X$ with $\Vert g\Vert< 1$ such that $g=f$ on $F$ and $g=0$ on $M\setminus B(0,R)$, for certain $R>0$.
\end{lemma}
\begin{proof}
    By Proposition \ref{prop:idenlocalycero} we may assume that $f(M)\subset B_X(0,r)$ and $F\subset B_M(0,r)$ for some $r>0$. Since $\|f\|<1$ there must be $R>0$ satisfying $\frac{r}{R-r}+\|f\|<1$. Let us then define $\widetilde\lambda:[0,\infty)\to[0,1]$ as
    $$\widetilde\lambda(t)=\begin{cases}1\;\;,&\text{ if }t\leq r,\\
    \frac{R-t}{R-r}\;\;,&\text{ if }r\leq t\leq R,\\
    0\;\;,&\text{ elsewhere.}\end{cases}$$
    Consider now $\lambda:M\to [0,1]$ as $\lambda=\widetilde\lambda\circ d(0,\cdot)$ and $g=\lambda\cdot f$. It is clear that $g=f$ on $B_M(0,r)\supset F$ and $g=0$ on $M\setminus B_M(0,R)$ so that it suffices to compute
    $$\|g\|\leq\|\lambda\|\sup f+\|f\|\sup\lambda\le\frac{r}{R-r}+\|f\|<1.$$
    The last computation holds since clearly $\|\lambda\|\le\|\widetilde\lambda\|\le\frac{1}{R-r}$ by Lemma \ref{l2}.
\end{proof}

Now we are ready to prove the following result.

\begin{theorem}\label{theo:octaunboununidis}
Let $M$ be a uniformly discrete unbounded metric space and let $X$ be a non-zero Banach space. Then the norm of $\mathcal F(M)\pten X$ is octahedral.
\end{theorem}

\begin{proof}
    As before let $\mu_1,\ldots, \mu_q\in \mathcal F(M)\pten X$ and $\varepsilon>0$, and let us find $\mu\in S_{\mathcal F(M)\pten X}$ such that $\Vert \mu_i+\mu\Vert>\frac{(1-\varepsilon)(\Vert \mu_i\Vert+1)-2\varepsilon}{1+\varepsilon}$.

By a density argument we can assume that, for every $1\leq i\leq q$, we can write $\mu_i:=\sum_{j=1}^{n_i}\lambda_{ij}\delta_{m_{ij}}\otimes x_{ij}$, for certain $n_i\in\mathbb N$, $\lambda_{ij}\in\mathbb R\setminus\{0\}, m_{ij}\in M\setminus \{0\}$ and $x_{ij}\in X\setminus\{0\}$. 

By Hahn-Banach theorem we can find, for every $i\in\{1,\ldots, q\}$, an element $f_i\in (\mathcal F(M)\pten X)^*=\Lip(M,X^*)$ such that $\Vert f_i\Vert< 1$ and $f_i(\mu_i)>(1-\varepsilon)\Vert \mu_i\Vert$. 

By Lemma \ref{lemma:densiunidiscre0} we can find, for every $1\leq i\leq q$, a Lipschitz function $g_i\in \Lip(M,X^*)$ with $\Vert g_i\Vert< 1$, such that $g_i=f_i$ on $\{m_{ij}: 1\leq j\leq n_i\}$ (in particular $g_i(\mu_i)=f_i(\mu_i)$) and $g_i=0$ on $M\setminus B(0,R)$ for $R>0$ big enough for every $1\leq i\leq q$. 

Select $x\in M$ satisfying that $\frac{R}{d(x,0)-R}<\varepsilon$, which is possible since $M$ is unbounded. Moreover, since $M$ is uniformly discrete there exists a $1$-Lipschitz function $h$ such that $h(z)=0$ if $z\neq x$. 

We claim that $\Vert g_i\pm h\Vert\leq 1+\varepsilon$ holds for $1\leq i\leq q$. Indeed, select $1\leq i\leq q$ and take $u,v\in M$ with $u\neq v$. Let us analyse $A:=\frac{\Vert g_i(u)\pm h(u)-(g_i(v)\pm h(v))\Vert}{d(u,v)}$.

Observe that if $u\neq x$ and $v\neq x$ then $A=\frac{\Vert g_i(u)-g_i(v)\Vert}{d(u,v)}< 1$. Otherwise assume, up to relabeling, that $v=x$, which implies $g_i(v)=0$. If $g_i(u)=0$ then it is immediate that $A=\frac{\Vert h(u)-h(v)\Vert}{d(u,v)}\leq \Vert h\Vert\leq 1$. Finally, if $g_i(u)\neq 0$ this implies $u\in B(0,R)$, from where we get
$$A\leq \frac{\Vert f_i(u)\Vert+\Vert h(u)-h(x)\Vert}{d(u,x)}\leq 1+\frac{\Vert f_i(u)-f_i(0)\Vert}{d(u,x)}\leq 1+\frac{d(u,0)}{d(u,x)}.$$
Now $d(u,x)=d(x,0)-d(u,0)\geq d(x,0)-R$. Moreover since $d(u,0)\leq R$ we infer that
$$1+\frac{d(u,0)}{d(u,x)}\leq 1+\varepsilon,$$
from where $\Vert g_i\pm h\Vert\leq 1+\varepsilon$ holds for $1\leq i\leq q$. 

Now consider $\mu\in S_{\mathcal F(M)\pten X}$ such that $h(\mu)>1-\varepsilon$. We claim now that $\Vert g_i(\mu)\Vert<2\varepsilon$ holds for every $i$. In fact, select $\sigma\in \{-1,1\}$ such that $\vert g_i(\mu)+\sigma h(\mu)\Vert=\vert g_i(\mu)\vert+\vert h(\mu)\vert$. Hence
$$1+\varepsilon\geq \Vert g_i+\sigma h\Vert\geq \vert g_i(\mu)+\sigma h(\mu)\vert=\vert g_i(\mu)\vert +\vert h(\mu)\vert>1-\varepsilon+\vert g_i(\mu)\vert,$$
so $\vert g_i(\mu)\vert\leq 2\varepsilon$ holds for $1\leq i\leq q$. On the other hand, since $h$ vanishes on $B(0,R)$ we infer that $h(\mu_i)=0$ holds for $1\leq i\leq q$. Hence
\[
\begin{split}(1+\varepsilon)\Vert \mu_i+\mu\Vert\geq (g_i+h)(\mu_i+\mu)& =g_i(\mu_i)+h(\mu)+g_i(\mu)\\
& \geq (1-\varepsilon)\mu_i+1-3\varepsilon\\
& \geq (1-\varepsilon)(\Vert \mu_i\Vert+1)-2\varepsilon,
\end{split}\]
so $\Vert \mu_i+\mu\Vert>\frac{(1-\varepsilon)(\Vert \mu_i\Vert+1)-2\varepsilon}{1+\varepsilon}$, as desired.

\end{proof}

To finish this section we will give a stronger property for the case in which $M'$ is infinite.

\begin{theorem}\label{theo:bidualfreeoh}
Let $M$ be a metric space such that $M'$ is infinite and let $X$ be a non-zero Banach space. Then the norm of $(\mathcal F(M)\pten X)^{**}$ is octahedral.
\end{theorem}

For the proof we will use that the norm of $(\mathcal F(M)\pten X)^{**}$ is octahedral if, and only if, $(\mathcal F(M)\pten X)^*=\Lip(M,X^*)$ has the SD2P (see Definition \ref{defi:SD2P}). We will indeed prove something stronger, which is the symmetric strong diameter two property. Even though we are not explicitly using this property, let us point out its definition for completeness: A Banach space $X$ is said to have the \textit{symmetric strong diameter two property (SSD2P)} if, for every $k\in\mathbb N$, every finite family of slices $S_1,\ldots, S_k$ of $B_X$, and every $\varepsilon>0$, there are $x_i\in S_i$ and there exists $\varphi\in B_X$ with $\Vert \varphi\Vert>1-\varepsilon$ such that $x_i\pm \varphi\in S_i$ for every $i\in\{1,\ldots, k\}$.

Observe that the SSD2P implies the SD2P, and the converse does not hold (see \cite{anp}). In spite of being a stronger property, SSD2P is sometimes easier to check than the SD2P. This happens in the case of infinite-dimensional uniform algebras or in Banach spaces with an infinite-dimensional centralizer (c.f. e.g. \cite{hlln} and references therein). It is particularly useful in Banach spaces for which there is not a good description of the dual Banach space (as spaces of Lipschitz functions), because \cite[Theorem 2.1]{hlln} gives a criterion for the SSD2P, which only makes use of weakly convergent nets. This idea was used in \cite[Lemma 2.3]{lr2020} where, combining this idea with the Lemma \ref{lemmaweaknull}, a criterion for the SSD2P on $\Lip(M,X)$ is given, and this criterion will be the key to proving the following theorem, which will give us Theorem \ref{theo:bidualfreeoh} by simply taking dual targets.

\begin{theorem}\label{theo:SSD2PLpschitz}
Let $M$ be a metric space such that $M'$ is infinite and let $X$ be a non-zero Banach space. Then $\Lip(M,X)$ has the SSD2P.    
\end{theorem}

\begin{proof}
Let $f_1,\ldots, f_k\in S_{\Lip(M,X)}$. Let $\{x_n\}\subseteq M'$ be a sequence of different cluster points. We can assume with no loss of generality that $f_i(x_n)\neq f_i(x_m)$ for every $n,m\in\mathbb N$ with $n\neq m$ and every $1\leq i\leq k$.

Consequently, for every $1\leq i\leq k$ and every $n\neq m$ we can find $r_n^i>0$ such that $\{B(f_i(x_n),r_n^i): n\in\mathbb N\}$ is a sequence of pairwise disjoint balls for every $n\in\mathbb N$. Now given $1\leq i\leq k$ and given $n\in\mathbb N$ we can find Lipschitz functions $\varphi_n^i:X\longrightarrow X$ such that $\Vert \varphi_n^i\Vert\rightarrow 1$ ($n\rightarrow \infty$) for every $i$, such that $\varphi_n^i(z)=z$ holds for every $z\in X\setminus B(f_i(x_n),r_n^i)$ and such that $\varphi_n^i(z)=f_i(x_n)$ holds for $z$ on a certain ball centred at $f_i(x_n)$.

Define $g_{n,i}:=\varphi_n^i\circ f_i:M\longrightarrow X$. Observe that $\lim \Vert g_{n,i}\Vert=1$. On the other hand observe that $\supp(g_{n,i}-f_i)\cap \supp(g_{m,i}-g)=\emptyset$ if $n\neq m$ for every $1\leq i\leq k$.

Now observe that $g_{n,i}$ is constant at an open ball centred at $x_n$ so we can find $r_n>0$ such that $g_{n,i}(x)=g_{n,i}(x_n)$ holds for every $x\in B(x_n,r_n)$ for every $n\in\mathbb N$ and every $1\leq i\leq k$.

Now given $n\in\mathbb N$ select $0<\beta_n<r_n$ such that $\frac{\beta_n}{r_n-\beta_n}\rightarrow 0$ and we construct a Lipschitz function $h_n:M\longrightarrow X$ with $\Vert h_n\Vert=1$ and $h_n(x)=0$ for every $x\in \{x_n\}\cup X\setminus B(x_n,\beta_n)$.

It is immediate that $\supp(h_n)\cap \supp(h_m)=\emptyset$ if $n\neq m$. Now we claim that $\Vert g_n^i\pm h_n\Vert\rightarrow 1$ for $1\leq i\leq k$.

Indeed, given $1\leq i\leq k$ and given $n\in\mathbb N$, we get by applying Lemma \ref{lemma:compuortolip} to $m=x_n$, $R=r_n$ and $\delta=\beta_n$ that $\Vert g_n^i\pm h_n\Vert\leq \Vert g_{n,i}\Vert(1+\frac{2\beta_n}{r_n-\beta_n})\rightarrow 1$.

Now we have that the sequences $(g_{n,i})$ and $(h_n)$ satisfy the requirements of \cite[Lemma 2.4]{lr2020}, so applying the above result we get that $\Lip(M,X)$ has the SSD2P, and the proof is finished.\end{proof} 

\section{Universally octahedral domains}\label{section:universaloctaedral}

In \cite{rz2023} the question of when a Banach space $X$ satisfies that $L(Y,X)$ is octahedral for every Banach space $Y$ is addressed (these spaces are called \textit{universally octahedral}, \cite[Definition 3.1]{rz2023}). The following result is proved.

\begin{theorem}\label{theo:carauniocta}
Let $X$ be a Banach space. The following are equivalent:
\begin{enumerate}
\item For every Banach space $Y$ and every subspace $H\subseteq L(Y,X)$ containing the finite-rank operators, the space $H$ is octahedral.
\item $X$ is universally octahedral.
\item For every finite dimensional Banach space $Y$, the space $L(Y,X)$ is octahedral.
\item For every finite dimensional uniformly convex Banach space $Y$, the space $L(Y,X)$ is octahedral.
\item For every $1<p<\infty$ and every $n\in\mathbb N$ the space $L(\ell_p^n,X)$ is octahedral.
\item For every $\varepsilon>0$, for every finite dimensional subspace $Z$ of $X$ and for every $n\in\mathbb N$, there exists an element $T:\ell_\infty^n\longrightarrow X$ with $\Vert T\Vert\leq 1$ and such that
$$\Vert z+T(y)\Vert\geq (1-\varepsilon)(\Vert z\Vert+\Vert y\Vert)$$
holds for every $y\in \ell_\infty^n$ and every $z\in Z$.
\end{enumerate}
\end{theorem}

A natural question in view of the previous result is to ask for ``universal octahedral domains'' in the following sense: when a Banach space $X$ satisfies that $L(X,Y)$ is octahedral for every Banach space $Y$?

The following gives a complete characterisation.

\begin{theorem}\label{theo:caraunivdomain}
    Let $X$ be a Banach space. The following assertions are equivalent:
\begin{enumerate}
    \item For every Banach space $Y$, $H$ is octahedral for every subspace $H\subseteq L(X,Y)$ with $\mathcal F(X,Y)\subseteq H$.
    \item For every Banach space $Y$, $L(X,Y)$ is octahedral.
    \item $X^*$ is universally octahedral, i.e. for every Banach space $Y$, $L(Y,X^*)$ is octahedral.
    \item For every $\varepsilon>0$, for every finite dimensional subspace $Z$ of $X^*$ and for every $n\in\mathbb N$, there exists an operator $T:\ell_\infty^n\longrightarrow X^*$ with $\Vert T\Vert\leq 1$ and such that
$$\Vert z+T(y)\Vert\geq (1-\varepsilon)(\Vert z\Vert+\Vert y\Vert)$$
holds for every $y\in \ell_\infty^n$ and every $z\in Z$.    
\end{enumerate}
\end{theorem}

\begin{proof}
(1)$\Rightarrow$(2) is immediate, and (3)$\Leftrightarrow$(4) is precisely Theorem \ref{theo:carauniocta}. To prove (2)$\Rightarrow$(3), take a Banach space $Y$. We want to prove that $L(Y,X^*)$ is octahedral. But this is immediate since the identification $L(Y,X^*)=L(X,Y^*)$ is isometric (c.f. e.g. \cite[pp. 24]{ryan}), so the conclussion is immediate from (2).

Let us prove (3)$\Rightarrow$(1). To this end take $Y$ an arbitrary Banach space and $H\subseteq L(X,Y)$ containing $\mathcal F(X,Y)$. In order to prove that $H$ is octahedral, take $T_1,\ldots T_n\in S_H$ and $\varepsilon>0$, and let us find $T\in S_H$ such that
$$\Vert T_i+T\Vert>2-\varepsilon.$$
To this end, consider adjoint operators $T_i^*:Y^*\longrightarrow X^*$ for every $1\leq i\leq n$, which are norm-one operators since taking adjoint is an isometric action. Since $X^*$ is universally octahedral, an inspection in the proof of \cite[Theorem 3.6]{rz2023} allows to find a finite rank operator $G:Y^*\longrightarrow X^*$ with $\Vert G\Vert=1$ and $\Vert T_i^*+G\Vert>2-\varepsilon$ for $1\leq i\leq n$. For every $i\in\{1,\ldots, n\}$, find an element $y_i^*\in S_{Y^*}$ such that 
$$\Vert T_i^*(y_i^*)+G(y_i^*)\Vert>2-\varepsilon\ \forall 1\leq i\leq n.$$
Let $\eta>0$. Applying \cite[Theorem 2.5]{op2007} for $F:=\spann\{y_1^*,\ldots, y_n^*\}\subseteq Y^*$ and $\eta$, there exists an operator $T\in \mathcal F(X,Y)\subseteq H$ such that $\vert \Vert T\Vert-\Vert G\Vert\vert<\eta$ and
$$T^*(y_i^*)=G(y_i^*)$$
holds for every $1\leq i\leq n$. In particular, given $1\leq i\leq n$, we have
$$2-\varepsilon<\Vert T_i^*(y_i^*)+G(y_i^*)\Vert=\Vert T_i^*(y_i^*)+T^*(y_i^*)\Vert\leq \Vert T_i^*+T^*\Vert=\Vert T_i+T\Vert.$$
Since $1-\eta<\Vert T\Vert<1+\eta$ we have $T/\Vert T\Vert\in S_H$ and
\[\begin{split}
\left\Vert T_i+\frac{T}{\Vert T\Vert} \right\Vert\geq \Vert T_i+T\Vert-\left\Vert T-\frac{T}{\Vert T\Vert} \right\Vert& >2-\varepsilon-\Vert T\Vert \frac{\vert\Vert T\Vert-1\vert}{\Vert T\Vert}\\
& >2-\varepsilon-(1+\eta)\frac{\eta}{1-\eta},
\end{split}\]
 and the conclusion follows from the arbitrariness of $\varepsilon$ and $\eta$.
\end{proof}

%A particular consequence in Lipschitz free spaces is the following.

%\begin{theorem}\label{theo:solulr2020}
%Let $M$ be a metric space such that $\Lip(M)$ is octahedral and $X$ be a non-trivial Banach space. Then $\Lip(M,X)$ is octahedral.
%\end{theorem}

%\begin{proof}
%Observe that the identification $\Lip(M,X)=L(\mathcal F(M),X)$ is isometric. Now $L(\mathcal F(M),X)$ is octahedral for every $X$ if, and only if, $L(Y,\Lip(M))$ is octahedral for every Banach space $Y$ in virtue of the above theorem, which is in turn equivalent to the fact that $\Lip(M)$ is octahedral by \cite[Theorem 3.1]{lr2020}, as desired.
%\end{proof}

%\begin{remark}
%The previous theorem gives a solution to \cite[Problem 1.3]{lr2020}, completing the partial solution given in \cite[Theorem 3.1]{lr2020} for dual codomains.
%\end{remark}

\section*{Acknowledgements}  During the celebration of the congress  ``$\AE$asy to define, hard to analyse : First conference on Lipschitz free spaces'' celebrated in Besan\c con (France) in September 2023, Richard Smith kindly left us Lemma \ref{loca0eident} together with its proof. The authors are deeply grateful to him because this lead to a major improvement of the results with respect to the first preprint version of this manuscript.

The authors also thank Gin\'es L\'opez-P\'erez for fruitful conversations on the topic of the paper.

This work was supported by MCIN/AEI/10.13039/501100011033: Grant PID2021-122126NB-C31 and by Junta de Andaluc\'ia: Grants FQM-0185 and PY20\_00255.
The research of Rub\'en Medina was also supported by FPU19/04085 MIU (Spain) Grant, by GA23-04776S project (Czech Republic) and by SGS22/053/OHK3/1T/13 project (Czech Republic). The research of Abraham Rueda Zoca was also supported by Fundaci\'on S\'eneca: ACyT Regi\'on de Murcia grant 21955/PI/22 and by Generalitat Valenciana project CIGE/2022/97.


\begin{thebibliography}{999999}

\bibitem {abrahamsen} T.~A. Abrahamsen, \textit{Linear extensions, almost isometries, and diameter two}, Extracta Math. \textbf{30}, 2 (2015), 135--151.

\bibitem {aln2}  T.~A. Abrahamsen, V. Lima, and O. Nygaard, \textit{Almost isometric ideals in Banach spaces}, Glasgow Math. J. \textbf{56} (2014), 395--407.

\bibitem {anp} T. A. Abrahamsen, O. Nygaard and M. P\~oldvere, \textit{New applications of extremely regular functions spaces}, Pacific J. Math. \textbf{301}, 2 (2019), 385--394.

\bibitem {am19} A. Avil\'es and G. Mart\'inez-Cervantes, \textit{Complete metric spaces with property (Z) are length spaces}, J. Math. Anal. Appl. \textbf{473}, 1 (2019), 334--344.

\bibitem{blr14} J.~Becerra Guerrero, G.~L\'opez-P\'erez and Abraham Rueda Zoca, \textit{Octahedral norms and convex combination of slices in Banach spaces}, J. Funct. Anal. \textbf{266} (2014), 2424--2435.

\bibitem {blr18} J.~Becerra Guerrero, G.~L\'opez-P\'erez and A.~Rueda Zoca, \textit{Octahedrality in Lipschitz-free Banach spaces}, Proc. R. Soc. Edinb. Sect. A Math. \textbf{148A}, 3 (2018), 447--460. 

\bibitem {bl00} Y.~Benyamini, J.~Lindenstrauss, \textit{Geometric Nonlinear Functional Analysis, vol. 1}, Amer. Math. Soc. Colloq. Publ., vol. 48, American Mathematical Society, Providence, RI, 2000.

\bibitem{bbi} D. Burago, Y. Burago and S. Ivanov, \textit{A Course in Metric Geometry}, Graduate Studies in Mathematics, Vol. 33, Amer. Math. Soc., Providence, 2001.

\bibitem {ccgmr19} B. Cascales, R. Chiclana, L. Garc\'ia-Lirola, M. Mart\'in and A. Rueda Zoca, \textit{On strongly norm attaining Lipschitz maps}, J. Funct. Anal. \textbf{277} (2019), 1677--1717.


\bibitem {dau} I.~K.~Daugavet, \textit{On a property of completely continuous operators in the space C}, Uspekhi Mat.
Nauk \textbf{18} (1963), 157-158 (Russian).

\bibitem {checos01} M. Fabian, P. Habala, P. H\'ajek, V. Montesinos, J. Pelant, and V. Zizler, \textit{Functional Analysis and Infinite
dimensional Geometry}, CMS Books in Mathematics, Springer-Verlag, New York, 2001.

\bibitem {fjt} T. Figiel, W. B. Johnson and L.Tzafriri, \textit{On Banach lattices and spaces having local unconditional structure, with applications to Lorentz function spaces}, J. Approx. Theory \textbf{13} (1975), 395--412.

\bibitem {gpr18} L. Garc\'ia-Lirola, A. Proch\'azka and A. Rueda Zoca, \textit{A characterisation of the Daugavet property in spaces of Lipschitz functions}, J. Math. Anal. Appl. \textbf{464} (2018), 473--492.

\bibitem{gks}
G.~Godefroy, N.~J. Kalton, and P.~D. Saphar, \emph{Unconditional ideals in
  {Banach} spaces}, Studia Math. \textbf{104} (1993), 13--59.

\bibitem {hlln} R.~Haller, J.~Langemets, V.~Lima and R.~Nadel, \textit{Symmetric strong diameter two property}, Mediterr. J. Math. \textbf{16} 
(2019), 16--35.

\bibitem {hhw} P. Harmand,  D. Werner  and W. Werner, {\it $M$-ideals in Banach
spaces  and Banach  algebras}, Lecture  Notes  in Math. {\bf 1547}, Springer-Verlag, Berlin-Heidelberg, 1993.

\bibitem {ikw} Y.~Ivakhno, V.~Kadets and D.~Werner, \textit{The Daugavet property for spaces of Lipschitz functions}, Math. Scand. \textbf{101} (2007), 261-279.

\bibitem {kkw} V.~Kadets, N.~Kalton and D.~Werner, \emph{Remarks on
    rich subspaces of Banach spaces}, Studia Math. \textbf{159} (2003), no.~2, 195--206.

\bibitem {kssw01} V. Kadets, R. V. Shvidkoy, G. G. Sirotkin, and D. Werner. \textit{Banach spaces with the Daugavet property}, Trans. Am. Math. Soc. \textbf{352}, 2 (2000), 855--873.

\bibitem {kar} L.A.~Karlovitz, \textit{The Construction and Application of Contractive Retractions in 2-Dimensional Normed
Linear Spaces}, Indiana Univ. Math. J. \textbf{22} (1972) 473-481.

\bibitem{langemetsthesis} J.~Langemets, \emph{Geometrical structure in diameter 2 Banach spaces}, Dissertationes Mathematicae Universitatis Tartuensis 99 (2015),
\url{http://dspace.ut.ee/handle/10062/47446}.

\bibitem {llr} J.~Langemets, V.~Lima and A.~Rueda Zoca, \textit{Almost square and octahedral norms in tensor products of Banach spaces}, RACSAM \textbf{111}, 3 (2017), 841--853.

\bibitem {llr2}  J.~Langemets, V.~Lima and A.~Rueda Zoca, \textit{Octahedral norms in tensor products of Banach spaces}, Q. J. Math. \textbf{68}, 4 (2017), 1247--1260.

\bibitem {lr2020} J. Langemets and A. Rueda Zoca, \textit{Octahedral norms in duals and biduals of Lipschitz-free spaces}, J. Funct. Anal. \textbf{279} (2020), article 108557.

\bibitem{litz}  J. Lindenstrauss and L. Tzafriri, \textit{Classical Banach Spaces II}, Springer-Verlag, 1979.

\bibitem{meny} P.~Meyer-Nieberg, \textit{Banach lattices}, Springer-Verlag, 1991.

\bibitem {op2007} E.~Oja and M.~P{\~o}ldvere, \emph{Principle of local reflexivity revisited},   Proc. Amer. Math. Soc. \textbf{135} (2007), no.~4, 1081--1088.

\bibitem {pirkthesis} K. Pirk, \textit{Diametral diameter two properties, Daugavet-, and $\Delta$-points in Banach spaces}, Dissertationes Mathematicae
Universitatis Tartuensis \textbf{133} (2020), \url{https://dspace.ut.ee/handle/10062/68458}.

\bibitem {rz2023} A. Rueda Zoca, \textit{Banach spaces which always produce octahedral spaces of operators}, Collect. Math. (2023). \url{https://doi.org/10.1007/s13348-023-00394-9}.

\bibitem {ruedathesis} A. Rueda Zoca, \textit{Geometry of Banach spaces with diameter two properties and octahedral norm,} PhD Thesis, Granada: Universidad de Granada, 2019. \url{http://hdl.handle.net/10481/57262}.

\bibitem {rueda22} A. Rueda Zoca, \textit{The Daugavet property in spaces of vector-valued Lipschitz functions}, J. Funct. Anal. \textbf{286}, 2 (2024), article 110208.

\bibitem {ryan} R.~A.~Ryan, \emph{Introduction to tensor products of Banach spaces}, Springer Monographs
in Mathematics, Springer-Verlag, London, 2002.

\bibitem {shv00} R. Shvidkoy, \textit{Geometric aspects of the Daugavet property}, J. Funct. Anal. \textbf{176} (2000), 198--212.

\bibitem {wer} D.~Werner, \textit{Recent progress on the Daugavet property}, Irish Math. Soc. Bull. \textbf{46} (2001) 77-97.

\end{thebibliography}
\end{document}